\documentclass{article}

\usepackage{amssymb}
\usepackage{amsmath,amsfonts,amsthm,amstext}
\usepackage{dsfont}
\usepackage[english]{babel}
\usepackage[utf8]{inputenc}
\usepackage[all]{xy}
\usepackage[pdftex]{hyperref}
\usepackage[pdftex]{color,graphicx}
\usepackage{mathrsfs}
\usepackage{appendix}
\usepackage{enumerate}
\usepackage{todonotes}
\usepackage{csquotes}
\usepackage{textcomp}
\usepackage{yfonts}
\usepackage{mathtools}

\theoremstyle{plain}
\newtheorem{theorem}{Theorem}

\newtheorem{lemma}[theorem]{Lemma}

\title{Representing integers as sums of mixed powers of primes}

\author{
  Geovane Matheus Lemes Andrade$^{a,*}$ \\
  \texttt{gmla.andrade@gmail.com}
  \and
  Hemar Godinho$^{a}$ \\
  \texttt{hemar@unb.br}
}

\date{}

\begin{document}
\maketitle

\begin{center}
$^{a}$Department of Mathematics, University of Brasília, Brasília, DF, Brazil\\
$^{*}$Corresponding author: \texttt{gmla.andrade@gmail.com}
\end{center}

\renewcommand{\thefootnote}{}
\footnotetext{\textit{Mathematics Subject Classification:} Primary 11P55; Secondary 11P32.}
\footnote{\textit{Key words and phrases:} Waring–Goldbach problem, circle method, additive representations.}
\renewcommand{\thefootnote}{\arabic{footnote}}

\begin{abstract}
    We establish two new Waring–Goldbach type representations: every sufficiently large odd integer $n$ can be expressed as 
    $$n = p_1^2+p_2^2+p_3^3+p_4^3+p_5^5+p_6^6+p_7^c,$$
    where each $p_i$ is prime and $c \in\{6,7\}$.
    
\end{abstract}

\section{Introduction}
    The Waring-Goldbach problem for mixed powers investigates the representation of natural numbers $n$ as a sum
    \begin{align}\label{equa1}
        n=x_1^{k_1}+x_2^{k_2}+\dots+x_s^{k_s},
    \end{align}
    where $x_1,\dots,x_s$ are prime variables and $ 1\leq k_1 \leq\dots \leq k_s$ are fixed integers. Relatively little is known about results of this kind. 
    Building on classical results of the form \eqref{equa1} with integer variables, several authors have investigated analogous representations in which the variables are restricted to prime numbers; see, for example, \cite{Bru23}, \cite{Bru21}, \cite{LICai2016116}, \cite{LuMu2016}, \cite{MU2014109}. Our work continues in this direction. 
    
    In 1987, Br\"udern \cite{Bru87} proved that every sufficiently large integer $n$ can be represented in the form 
    $$n = x_1^2+x_2^2+x_3^3+x_4^3+x_5^a+x_6^b+x_7^c,$$
    where $x_1,\dots,x_7$ are positive integers and $c\geq b\geq a\geq 3$ are fixed integers satisfying $1/a +1/b +1/c>\frac{1}{3}$. We prove the Waring–Goldbach analogue of Brudern’s result for $(a,b,c)=(5,6,6),(5,6,7)$.

\begin{theorem}\label{teo1}
    Let $6\leq c\leq 7$. Then, for every sufficiently large odd integer $n$, the equation
    \begin{align}\label{equa2}
        n = x_1^2+x_2^2+x_3^3+x_4^3+x_5^5+x_6^6+x_7^c
    \end{align}
    has a solution with all $x_j$ prime.
\end{theorem}

We follow a similar approach as Brüdern \cite{Bru23}. A suitable choice of the parameter in the definition of an exponential sum enables us to obtain sharper bounds for mean values, thanks to Davenport’s method \cite{davenport42} in the refined form developed by Vaughan \cite{Vau86}. Applications of Hölder’s inequality yield mean values that are well understood, leaving a fractional contribution from the sixth (or seventh) power to be estimated. This estimation is carried out using Kumchev’s bounds for exponential sums over primes \cite{Kum06}, along with recent advances in Vinogradov’s mean value theorem as exploited in \cite{KumWoo16}. A pruning argument then completes the proof of the theorem.

\section{Notation and preliminaries}

Throughout this paper, $p$ denotes a prime number, $e(\alpha)$ abbreviates $\exp(2\pi i \alpha)$, and $\phi(q)$ is Euler's totient function.
 The implicit constants in Vinogradov’s and Landau’s symbols will depend on the value assigned to $\epsilon$.

Fix $6 \leq c \leq 7$, and let $n$ be a large positive integer. Define 
$$P_k = n^{1/k}, \qquad L = \log n.$$
For $k = 2,3,6,c$, let $$f_k(\alpha) = \sum \limits_{\frac{1}{2}P_k < p \leq P_k} e(\alpha p^k)\log p.$$
Moreover, set $\theta_6 = 5/6 $ and $\theta_7=87/98$, and define $$f_{5,c}(\alpha) = \sum \limits_{1 < p \leq P_5^{\theta_c}} e(\alpha p^5)\log p.$$
For simplicity, we often write $f_k$ for $f_k(\alpha)$, and introduce $$F_c(\alpha) = f_2(\alpha)^2f_3(\alpha)^2f_{5,c}(\alpha)f_6(\alpha)f_c(\alpha).$$

By orthogonality, the integral
\begin{align}\label{ort}
    \nu_c(n) = \int_{0}^{1} F_c(\alpha)e(-\alpha n)d\alpha
\end{align}
counts the number of solutions of \eqref{equa2} with prime variables constrained to certain intervals, and  weighted by $\log x_1 \cdots \log x_7$. Let 
$$\Theta_c = \frac{2}{3} + \frac{\theta_c}{5}+\frac{1}{6}+\frac{1}{c}.$$
Our goal is to prove that the lower bound $$\nu_c(n) \gg n^{\Theta_c}$$ holds for all sufficiently large odd integers $n$.

    Fix a real number $B \geq 1$. Let $\mathfrak{M}$ be the union of all intervals 
    \begin{align}\label{mqamajor}
        \mathfrak{M}(q,a) = \{\alpha \in[0,1]: |\alpha - a/q| \leq L^B/n \},
    \end{align} 
    with $0\leq a \leq q$, $(a,q)=1$ and $1\leq q \leq L^B$. Let $\mathfrak{m} = [0,1]\setminus \mathfrak{M}$. Theorem \ref{teo1} will be proved if we show that for all sufficiently large odd $n$, 
    \begin{align}\label{task}
        \int \limits_{\mathfrak{M}} F_c(\alpha)e(-\alpha n)d\alpha \gg n^{\Theta_c}, \qquad \int \limits_{\mathfrak{m}} F_c(\alpha)e(-\alpha n)d\alpha \ll n^{\Theta_c}L^{-1}.
    \end{align}
    We begin by analyzing the  contribution from $\mathfrak{M}$.

\section{Major arcs}

For $k=2,3,5,6,c$, let 
\begin{align}\label{vkbeta}
    S_k(q,a) = \sum \limits_{\substack{x=1 \\ (x,q)=1}}^{q} e\left(\frac{ax^k}{q} \right), \qquad v_k(\beta) = \frac{1}{k}\sum \limits_{2^{-k}n < m \leq n} m^{1/k - 1} e(\beta m).
\end{align}

Building on Hua’s proof of Lemma 6 in \cite{Hua38}, we establish the following result. 

\begin{lemma}\label{hua6}
    Let $a,q$ be integers satisfying $0\leq a \leq q$, $(a,q)=1$ and $1\leq q \leq L^B$. Let $k \ne 5$. Then, for $\alpha \in \mathfrak{M}(q,a)$, $$f_k(\alpha) = \frac{S_k(q,a)}{\phi(q)}v_k(\alpha -a/q) + O(P_ke^{-c_1\sqrt{L}}),$$
    where $c_1>0$ is an absolute constant. Moreover,
    $$f_{5,c}\left(\frac{a}{q}\right) = \frac{S_5(q,a)}{\phi(q)}P_5^{\theta_c} + O(P_5^{\theta_c}e^{-c_2\sqrt{L}}),$$ 
    where $c_2>0$ is an absolute constant.
\end{lemma}

\begin{lemma}\label{major}
    For large odd $n$ one has $$ \int \limits_{\mathfrak{M}} F_c(\alpha)e(-\alpha n)d\alpha \gg n^{\Theta_c}.$$
\end{lemma}

\begin{proof}
   We begin by providing a suitable approximation for $f_{5,c}(\alpha)$. Suppose $\alpha \in \mathfrak{M}(q,a)$ with $q \leq L^B$ and $(a,q)=1$. For $p\leq P_5^{\theta_c}$ and $|\beta| \leq L^B/n$, Taylor expansion yields $$e(\beta p^5) = 1 + O(L^Bn^{-1}P_5^{5\theta_c}).$$ Taking $\beta = \alpha - a/q$ and noticing that $n^{-1}P_5^{5\theta_c} = n^{\theta_c -1} $, we conclude from the definition of $f_{5,c}$ that $$ f_{5,c}(\alpha) = f_{5,c}(a/q) + O(P_5^{\theta_c}L^Bn^{\theta_c -1}).$$
   Using Lemma \ref{hua6}, we infer 
   \begin{align}\label{f5alpha}
       f_{5,c}(\alpha) = \phi(q)^{-1}S_5(q,a)P_5^{\theta_c} + O(P_5^{\theta_c}L^{-4B}).
   \end{align}
   Multiplying the approximations for $f_k(\alpha)$ yields an estimate of $F_c(\alpha)$. Let us define
   \begin{align}
       U_c(q,a) &= S_2(q,a)^2S_3(q,a)^2S_5(q,a)S_6(q,a)S_c(q,a), \nonumber \\
       w_c(\beta) & = v_2(\beta)^2v_3(\beta)^2v_6(\beta)v_c(\beta). \label{wcdef}
   \end{align}
     Using \eqref{f5alpha} and the approximations from Lemma \ref{hua6}, we obtain $$F_c(\alpha) = P_5^{\theta_c} \phi(q)^{-1}U_c(q,a)w_c(\alpha - a/q) + O(n^{1+\Theta_c}L^{-4B}).$$
     We integrate this over $\mathfrak{M}$, a set of measure $O(L^{3B}/n)$, yielding
     \begin{align}
         \int \limits_{\mathfrak{M}}F_c&(\alpha)e(-\alpha n) d\alpha \nonumber \\&= P_5^{\theta_c} \sum \limits_{q \leq L^B}\frac{A_{n,c}(q)}{\phi(q)^7} \int_{-L^B/n}^{L^B/n} w_c(\beta)e(-\beta n)d\beta + O(n^{\Theta_c} L^{-B}), \label{over}
     \end{align} 
     where 
     $$A_{n,c}(q) = \sum \limits_{\substack{a=1 \\ (a,q)=1}}^{q} U_c(q,a)e(-an/q).$$
    By Lemma 5 of \cite{Hua38} we have $U_c(q,a) \ll q^{7/2 + \epsilon}$ for coprime $a,q$, and hence $A_{n,c}(q) \ll q^{9/2+\epsilon}$ uniformly in $n$. It follows that the series $$ \mathfrak{S}_c(n) = \sum \limits_{q =1}^{\infty}\frac{A_{n,c}(q)}{\phi(q)^7}$$ converges absolutely. Furthermore, routine calculations show that the estimate 
    \begin{align*}
        \sum \limits_{q \leq L^B}\frac{A_{n,c}(q)}{\phi(q)^7} = \mathfrak{S}_c(n) + O(L^{-B})
    \end{align*}
    holds uniformly in $n$. Also, by Lemma 6.2 of \cite{bookVau}, one has
    \begin{align}\label{wcestimate}
        w_c(\beta) \ll n^{\Theta_c +1-\frac{\theta_c}{5}}(1+n|\beta|)^{-2}.
    \end{align}
    Thus, we can replace the sum in \eqref{over} by $\mathfrak{S}_c(n)$ with negligible errors. 
    We aim to prove that $ \mathfrak{S}_c(n) \gg 1$ for all large odd $n$. Building on the argument used in the proof of Lemma 2.11 of \cite{bookVau}, we can likewise establish that $A_{n,c}(q)$ is multiplicative. Moreover, Lemma 4 of \cite{Hua38} shows that $A_{n,c}(q) = 0$ unless $q$ is square-free. This yields
    \begin{align}\label{prod}
        \mathfrak{S}_c(n) = \prod \limits_{p}(1+(p-1)^{-7}A_{n,c}(p)).
    \end{align}

    Let $M_{n,c}(p)$ denote the number of solutions of \eqref{equa2} in the finite field $\mathbb{F}_p$, with all variables non-zero. By orthogonality $\mod p$, $$1 + (p-1)^{-7} A_{n,c}(p)= p(p-1)^{-7}M_{n,c}(p).$$
    Applying the Cauchy–Davenport theorem (Lemma 2.14 of \cite{bookVau}) gives  $M_{n,c}(p) \geq 1$, except when $p=2$ and $2 \mid n$. Hence the factors in \eqref{prod} are positive for all odd $n$, and the uniform convergence with respect to $n$ gives us a number $\ell$ such that, for odd $n$,
    $$\mathfrak{S}_c(n) \geq \frac{1}{2} \prod \limits_{p\leq \ell}(1+(p-1)^{-7}A_{n,c}(p)) \geq \frac{1}{2} \prod \limits_{p\leq \ell}p^{-6}.$$
    Thus, the lower bound $ \mathfrak{S}_c(n) \gg 1$ holds for all large odd $n$.

    To conclude the proof, it suffices to show that the integral on the right hand side of \eqref{over} is $\gg n^{\Theta_c-\frac{\theta_c}{5}}$. Using the estimate \eqref{wcestimate}, we infer 
    $$ \int_{-L^B/n}^{L^B/n} w_c(\beta)e(-\beta n)d\beta = \int_{-1/2}^{1/2} w_c(\beta)e(-\beta n)d\beta + O(n^{\Theta_c-\frac{\theta_c}{5}}L^{-B}).$$
    By orthogonality (see \eqref{wcdef} and \eqref{vkbeta}), the integral on the right hand side is
    \begin{align}\label{sum}
        \frac{1}{2^2 \cdot 3^2 \cdot 6\cdot c} \sum (m_1m_2)^{-\frac{1}{2}}(m_3m_4)^{-\frac{2}{3}}(m_6)^{-\frac{5}{6}}(m_7)^{\frac{1-c}{c}},
    \end{align}
    where the sum extends over $m_1,m_2,m_3,m_4,m_6,m_7$ subject to 
    \begin{align}\label{sixsum}
        m_1+m_2+m_3+m_4+m_6+m_7=n,
    \end{align}
    with
    \begin{align}
        & 2^{-2}n<m_1,m_2 \leq n, \qquad 2^{-3}n<m_3,m_4 \leq n, \qquad \label{restr1}\\
        &2^{-6}n<m_6\leq n, \hspace{0.5cm} 2^{-c}n<m_7\leq n. \label{restr2}
    \end{align}
    Observe that whatever $m_1,m_2 \leq \frac{9}{32}n$,  $m_3,m_4 \leq \frac{9}{64}n$ and $m_6\leq \frac{3}{128}n$ in accordance with \eqref{restr1} and \eqref{restr2}, we can solve \eqref{sixsum} with $m_7$ satisfying \eqref{restr2}. Consequently, the sum in \eqref{sum} is bounded below by $n^{\Theta_c-\frac{\theta_c}{5}}$. This completes the proof of our lemma.
    
\end{proof}

\section{Minor arcs}

In order to bound the minor arcs contribution to \eqref{ort}, we consider the following mean values:
\begin{align*}
    J_{1,c} = \int_{0}^{1} |f_2f_c|^2|f_{5,c}|^4d\alpha, \qquad
    J_2 = \int_{0}^{1}|f_2f_3f_6|^2d\alpha,\qquad
    J_3= \int_{0}^{1}|f_2|^2|f_3|^4d\alpha.
\end{align*}
We aim to show that
\begin{align}
    J_{1,c} &\ll P_2^{1+\epsilon}P_cP_5^{2\theta_c},\label{j1}\\
    J_2 &\ll n^{1+\epsilon},\label{j2}\\
    J_3 & \ll n^{\frac{4}{3}+\epsilon}.\label{j3}
\end{align}

To estimate $J_1$, we consider the number $T_c$ of solutions of the equation $$y_1^c-y_2^c = z_1^5+z_2^5-z_3^5-z_4^5,$$ with the variables constrained by
\begin{align}\label{yz}
    \frac{1}{2}P_c<y_i\le P_c, \qquad 1\leq z_j \leq P_5^{\theta_c} \qquad (1\leq i \leq 2, 1 \leq j \leq 4).
\end{align}
We apply Lemma 4 from \cite{Vau86}, taking $k=c, \lambda = \theta_c, P=\frac{1}{2}P_c, j=3$, with $R(m)$ denoting the number of pairs $(z_1,z_2) \in [1,P_5^{\theta_c}]^2$ such that $z_1^5 + z_2^5 = m$. In this setting, the quantity $T$ in Vaughan's lemma coincides with our $T_c$. Note that the equation $z_1^5 + z_2^5 = m$ implies $z_1 + z_5 \mid m$. Thus, by a divisor counting argument, we obtain $R(m) \ll m^{\epsilon}$. Consequently, 
$$ \sum \limits_{m}R(m) \ll P_5^{2\theta_c}, \qquad \sum \limits_{m}R(m)^2 \ll P_5^{2\theta_c + \epsilon}.$$
Given this setup, Lemma 4 of \cite{Vau86} yields the estimate
$$T_c \ll P_cP_5^{2\theta_c+\epsilon} + P_c^{\frac{7c}{8}(\theta_c-1)+\frac{11}{8}+\epsilon}(P_5^{2\theta_c})^{\frac{9}{8}} \ll P_c^{1+\epsilon}P_5^{2\theta_c}.$$
Here, the parameter $\theta_c$ is chosen to balance the contributions of the two terms in the bound for $T_c$. 

 By orthogonality, we have $J_{1,c} \ll N_cL^8$, where $N_c$ denotes the number of solutions of the equation
\begin{align}\label{bigequ}
    x_1^2-x_2^2=y_1^c-y_2^c+z_1^5+z_2^5-z_3^5-z_4^5
\end{align}
in integers satisfying \eqref{yz} and $\frac{1}{2}P_2 < x_1,x_2\leq P_2$. 
The number of solutions for which the right hand side of \eqref{bigequ} equals zero is $O(P_2T_c)$. To count the remaining solutions, we observe that there are $O(P_c^2P_5^{4\theta_c})$ choices for $y_j,z_i$ such that the right hand side of \eqref{bigequ} is nonzero. For each such choice, both $x_1-x_2$ and $x_1+x_2$ divide the right hand side of \eqref{bigequ}, yielding $O(P_2^\epsilon)$ possibilities for $x_1,x_2$. This gives 
$$N_c \ll P_2P_c^{1+\epsilon}P_5^{2\theta_c} + P_2^{\epsilon}P_c^2P_5^{4\theta_c} \ll P_2^{1+\epsilon}P_cP_5^{2\theta_c},$$
which completes the proof of \eqref{j1}.

As for the means $J_2$ and $J_3$, we observe that both integrals can be estimated by counting the number of solutions of an underlying Diophantine equation. As in the discussion of $J_1$, this approach removes the logarithmic weights and the restriction to prime variables. Therefore, to estimate $J_2$ and $J_3$, we may replace the $f_j$ with ordinary Weyl sums. By applying Lemma 1 from \cite{Bru87}, we obtain the bound stated in \eqref{j2}.

The quantity $J_3$ is bounded above by the number of solutions of the equation 
$$x_1^2 -x_2^2 = h_1^3+h_2^3-h_3^3-h_4^3,$$
with $1\leq x_1,x_2\leq P_2$ and $1\leq h_i \leq P_3$. By an application of Lemma 2.5 of \cite{bookVau} (Hua's lemma),
the number of choices for $h_i$ for the right hand side to be zero is $O(P_3^{2+\epsilon})$. For each such choice, there are $O(P_2)$ possibilities for $x_1,x_2$, yielding a total of $O(P_2P_3^{2+\epsilon})$ solutions. For the remaining solutions, we consider choices of $h_1,\dots,h_4$ such that the right hand side is non-zero. There are $O(P_3^4)$ such choices, and for each one, there are $O(P_2^\epsilon)$ possibilities for $x_1,x_2$. Thus, 
\begin{align*}
 J_3 \ll P_2P_3^{2+\epsilon} + P_2^\epsilon P_3^4 \ll n^{\frac{4}{3}+\epsilon}, 
\end{align*}
establishing \eqref{j3}.

We are now well equipped to handle the contribution to \eqref{ort} arising from the sets 
\begin{align}
    \mathcal{E}_6 &= \{ \alpha \in [0,1]: |f_6(\alpha)| \leq P_6^{1-24\tau}\}, \label{set1}\\
    \mathcal{E}_7 &= \{ \alpha \in [0,1]: |f_7(\alpha)| \leq P_7^{\frac{139}{140}-24\tau}\}, \label{set2}
\end{align}
where $\tau=2^{-20}$.

\begin{lemma}\label{epsilon}
    One has $$ \int \limits_{\mathcal{E}_c} F_c(\alpha)e(-\alpha n)d\alpha \ll n^{\Theta_c - \tau}.$$
\end{lemma}

\begin{proof}
    We apply H\"older's inequality to deduce that 
    $$ \int \limits_{\mathcal{E}_c} F_c(\alpha)e(-\alpha n)d\alpha \leq J_{1,c}^{1/4}J_2^{1/2}J_3^{1/4} \sup \limits_{\alpha \in \mathcal{E}_c}|f_c(\alpha)|^{1/2}.$$
    Using the bounds from \eqref{j1}, \eqref{j2}, \eqref{j3}, and the definitions in \eqref{set1} and \eqref{set2}, along with a straightforward calculation, the result follows.
\end{proof}

To treat the complement of $\mathcal{E}_c$, we apply Lemma 2.2 from \cite{KumWoo16} in conjunction with partial summation, taking $k=c$ and $X = \frac{1}{2}P_c$. It follows that $$[0,1]\setminus \mathcal{E}_c \subset \mathfrak{R}_c,$$
where $\mathfrak{R}_c$ is the union of the disjoint intervals
$$ \mathfrak{R}_c(q,a) = \{\alpha \in [0,1]: |q\alpha -a| \leq Q_cn^{-1}\}$$
with $0\leq a \leq q$, $(a,q)=1$ and $1 \leq q \leq Q_c$, where $Q_c = n^{\frac{2}{3c^2(c-1)}}$. 

Define a function $\Upsilon_c:\mathfrak{R}_c \to [0,1]$ by $$\Upsilon_c(\alpha) = (q+n|q\alpha - a|)^{-1} \qquad (\alpha \in \mathfrak{R}_c(q,a)).$$
Then, by Theorem 2 of \cite{Kum06} and partial summation, there exists a constant $C>0$ such that for $\alpha \in \mathfrak{R}_c$ and $k=2,3,6,c$, one has 
$$f_k(\alpha) \ll P_k L^C\Upsilon_c(\alpha)^{1/2 - \epsilon}.$$
Using the trivial bound for $f_{5,c}$, we deduce 
$$F_c(\alpha) \ll n^{1+\Theta_c}L^{6C}\Upsilon_c(\alpha)^{5/2}.$$
We now integrate over $\mathfrak{R}_c\setminus \mathfrak{M}$. Routine calculations yield (see \eqref{mqamajor})
$$ \int \limits_{\mathfrak{R}_c\setminus \mathfrak{M}} |F_c(\alpha)|d\alpha \ll n^{1+\Theta_c}L^{6C}\int \limits_{\mathfrak{R}_c\setminus \mathfrak{M}} |\Upsilon_c(\alpha)|^{5/2}d\alpha \ll  n^{\Theta_c}L^{6C-B/2}.$$
This establishes the following result.

\begin{lemma}\label{ult}
    Let $B\geq 12C + 2$. Then $$\int \limits_{\mathfrak{R}_c\setminus \mathfrak{M}} |F_c(\alpha)|d\alpha \ll  n^{\Theta_c}L^{-1}.$$
\end{lemma}

We are now in a position to complete the proof of Theorem \ref{teo1}. We choose $B$ so that  Lemma \ref{ult} applies. By the definitions of the sets $\mathfrak{M}, \mathfrak{R}_c $ and $\mathcal{E}_c$, we observe that $$\mathfrak{m}\subset \mathcal{E}_c \cup \mathfrak{R}_c\setminus \mathfrak{M}.$$
Therefore, combining Lemmas \ref{epsilon}, \ref{ult}, and \ref{major}, we obtain \eqref{task}, which completes the proof of Theorem \ref{teo1}.

\section*{Acknowledgements}

The authors would like to thank the Graduate Program in Mathematics at the University of Brasília for the academic support provided during the development of this research. We also acknowledge the financial support from CAPES (Coordenação de Aperfeiçoamento de Pessoal de Nível Superior, Brazil).

\section*{Declaration of Interest}
The authors declare that they have no known competing financial interests or personal relationships that could have appeared to influence the work reported in this paper.

\bibliographystyle{elsarticle-num}
\bibliography{sample}

\end{document}